\documentclass[a4paper, 11pt]{article}
\usepackage{amsmath, amssymb, amsthm, latexsym, enumerate}
\usepackage[alphabetic]{amsrefs}
\usepackage{xspace}
\usepackage{a4wide}
\usepackage[usenames,dvipsnames,svgnames]{xcolor}
\usepackage{tikz}
\usepackage{hyperref}
\usepackage{authblk}
\usepackage{chngcntr}
\usepackage{enumitem}

\counterwithin{figure}{section}

\newtheorem{btheorem}{Theorem}[]

\newtheorem{proposition}{Proposition}[section]
\newtheorem{theorem}[proposition]{Theorem}
\newtheorem{corollary}[proposition]{Corollary}
\newtheorem{lemma}[proposition]{Lemma}

\newtheorem*{theorem*}{Theorem}

\theoremstyle{definition}
\newtheorem{definition}[proposition]{Definition}

\newtheorem{remark}[proposition]{Remark}

\newtheorem*{question*}{Question}

\newcommand{\set}[1]{\left\{#1\right\}}
\newcommand{\setcon}[2]{\left\{#1\ \left|\ #2\right.\right\}}
\newcommand{\abs}[1]{\left\lvert#1\right\rvert}

\newcommand{\R}{\mathbb{R}}
\newcommand{\Z}{\mathbb{Z}}
\newcommand{\N}{\mathbb{N}}

\newcommand{\diam}{\textup{diam}}
\newcommand{\fpres}[2]{\left\langle #1 \left| #2 \right.\right\rangle}
\newcommand{\stabdim}{\textrm{asdim}_s}
\newcommand{\Cay}{\mathrm{Cay}}

\title{Relatively hyperbolic groups with fixed peripherals}

\author[1]{Matthew Cordes\thanks{The author was partially supported by NSF grant DMS-1106726}}

\author[2]{David Hume\thanks{The author was supported by the grant ANR-14-CE25-0004 ``GAMME''}}

\affil[1]{Brandeis University, Waltham, MA, USA}
\affil[2]{Facult\'{e} des Sciences d'Orsay, Universit\'{e} Paris-Sud,
F-91405 Orsay, France}

\date{\today}

\begin{document}

\maketitle

\begin{abstract}

We build quasi--isometry invariants of relatively hyperbolic groups which detect the hyperbolic parts of the group; these are variations of the stable dimension constructions previously introduced by the authors.

We prove that, given any finite collection of finitely generated groups $\mathcal{H}$ each of which either has finite stable dimension or is non-relatively hyperbolic, there exist infinitely many quasi--isometry types of one--ended groups which are hyperbolic relative to $\mathcal{H}$.

The groups are constructed using small cancellation theory over free products.
\end{abstract}

\section{Introduction}
In \cite{CordesHume_stable} we defined a new quasi--isometry invariant for geodesic metric spaces: the \textbf{stable asymptotic dimension}, $\stabdim$.

The purpose of this paper is to construct related invariants, which are more naturally suited to distinguishing relatively hyperbolic groups with the same peripheral subgroups.

In \cite{Gromov87}, Gromov introduced relatively hyperbolic groups as a generalisation of hyperbolic groups, this exposition was substantially developed by Farb and Bowditch \cites{Farb,BowditchRelHyp}. The class of relatively hyperbolic groups includes: hyperbolic groups, amalgamated products and HNN-extensions over finite subgroups, fully residually free (limit) groups \cites{Da03,Al05}---which are key objects in solving the Tarski conjecture \cites{Se01,KM10}, geometrically finite Kleinian groups and fundamental groups of non-geometric closed $3$-manifolds with at least one hyperbolic component \cite{Da03}.

Recall that a finitely generated group is relatively hyperbolic if it is hyperbolic relative to some finite collection of proper subgroups, and non-relatively hyperbolic (nRH) otherwise (see Definition $\ref{defn:relhyp}$).

Relative hyperbolicity is a quasi--isometry invariant \cite{Drutu-RHgpsGeomQI}. Moreover, in \cite{Behrstock-Drutu-Mosher} it is proved that hyperbolicity relative to non-relatively hyperbolic subgroups is quasi--isometrically rigid: that is, if $G$ is hyperbolic relative to a collection of non-relatively hyperbolic groups $\mathcal{H}$ and $q \colon G\to G'$ is a quasi--isometry, then $G'$ is hyperbolic relative to a collection of non-relatively hyperbolic groups $\mathcal{H}'$ where each $H'\in\mathcal{H}'$ is quasi--isometric to some $H\in\mathcal{H}$, in particular the image under $q$ of a coset of some $H\in \mathcal{H}$ is contained in a uniform neighbourhood of some $gH'$ with $H'\in\mathcal{H}'$.

The goal of our paper is to give one approach to answering the following question which appears in \cite{Behrstock-Drutu-Mosher}:

\begin{question*}
How may we distinguish non-quasi--isometric relatively hyperbolic groups with nRH peripheral subgroups when their peripheral subgroups are quasi--isometric?
\end{question*}

One obvious invariant is the number of relative ends, which for an infinite group is either $2$, when the group is virtually cyclic; $\infty$, when the group splits as an amalgamated product or HNN extension over a finite subgroup; or $1$, when no such splitting exists \cites{St68,St71}.

In the infinitely ended case, under some accessibility assumptions (for instance, finite presentability \cite{Dun85}), the group admits a graph of groups decomposition with vertex groups which are finite or $1$--ended and edge groups which are finite. These groups are quasi--isometric to free products, and quasi--isometries between free products are very well understood \cite{Pap-Whyte}.

Our focus is therefore mostly on $1$--ended relatively hyperbolic groups.

Of course, two one--ended groups which are hyperbolic relative to the same nRH subgroups need not be quasi--isometric: examples of Schwartz provide infinitely many non-quasi--isometric groups which are hyperbolic relative to $\Z^2$ \cite{Schw-QIlattices}. 

Another approach to the above question is highlighted by \cite[Theorem $6.3$ and Corollary $6.5$]{GroffBB}, where it is proved that the quasi--isometry type of the cusped (Bowditch) space is an invariant of relatively hyperbolic groups with nRH peripherals. Our invariants are different in character and have the advantage that they allow generalisations to other classes of peripherals.
\medskip

Given a collection $\mathcal{H}$ of finitely many finitely generated groups (possibly with repetitions) it is not clear that there even exist $1$--ended groups which are hyperbolic relative to $\mathcal{H}$, so our first task is to provide a general construction using small cancellation theory over free products.

\begin{btheorem}\label{bthm:make1endedRH} Let $\mathcal{H}$ be a collection of finitely many finitely generated groups (possibly with repetitions). There is a group $G$ which is $1$--ended and hyperbolic relative to $\mathcal{H}$.
\end{btheorem}

The proof of this theorem extends results in classical small cancellation theory by Champetier \cite{ChampSC}. To complete this we will show that a natural variant of Strebel's classification of geodesic triangles \cite{Str90} holds in the setting of small cancellation groups over free products.

\medskip
Our real goal is to produce invariants which allow us to deduce statements of the following form: under certain hypotheses on $\mathcal{H}$, there are infinitely many non-quasi--isometric $1$--ended groups which are hyperbolic relative to $\mathcal{H}$.

In \cite[Theorem I]{CordesHume_stable} we proved that the \textbf{stable dimension} of a relatively hyperbolic group is finite if and only if its peripheral subgroups have finite stable dimension. Our next result in this paper shows that in this case, the stable dimension is an interesting quasi--isometry invariant of relatively hyperbolic groups.

\begin{btheorem}\label{bthm:relhypstable} Let $\mathcal H$ be a finite collection of finitely generated groups each of which has finite stable dimension. There is an infinite family of $1$--ended groups $(G_n)_{n\in\N}$, with strictly increasing stable dimension, where each $G_n$ is hyperbolic relative to $\mathcal H$.
\end{btheorem}

In particular, the groups $G_n$ are pairwise non-quasi--isometric. Notice that without the $1$--ended assumption, the result easily follows from the existence of an infinite family of hyperbolic groups of unbounded asymptotic dimension and \cite{Pap-Whyte} by considering free products.

The construction uses a family of hyperbolic groups of unbounded asymptotic dimension $\set{H^n}_{n \in \N}$ and an application of Theorem \ref{bthm:make1endedRH} to $\mathcal{H}^n=\mathcal{H}\cup\set{H^n}$.

Since relatively hyperbolic groups are finitely relatively presented \cite{OsinRelhyp}, there are only a countably infinite number of finitely generated groups which are hyperbolic relative to a specified collection of peripherals, so this result is optimal.

The class of groups with finite stable dimension is very large: it includes all virtually solvable groups, all groups with finite asymptotic dimension and all groups which split as a direct product of two infinite groups \cite{CordesHume_stable}. There is currently no example of a finitely generated amenable non-virtually cyclic group with positive stable dimension.
\medskip

However, this result does not exactly address the question raised in \cite{Behrstock-Drutu-Mosher}, since there exist non-relatively hyperbolic groups with infinite stable dimension \cite{Gruber_personal}. To deal with this, we will introduce the notion of \textbf{relative stable dimension}: the supremal asymptotic dimension of a stable subset of a space $X$ which ``avoids'', in some sense, a collection of subspaces $\mathcal{Y}$ of $X$. We denote this $\stabdim(X;\mathcal{Y})$. This dimension is defined precisely in Section \ref{sec:avoidnRH}.

To ease notation, given a group $G$ and a collection of subgroups $\mathcal{H}$ of $G$, we write $LC(\mathcal{H})=\setcon{gH}{g\in G,\ H\in\mathcal{H}}$.

Under certain hypotheses the relative stable dimension is a quasi--isometry invariant.

\begin{btheorem}\label{bthm:relstabdimQI}
Let $G$ be a group which is hyperbolic relative to a collection of non-relatively hyperbolic subgroups $\mathcal{G}$. If $H$ is a group quasi--isometric to $G$, then for any collection of non-relatively hyperbolic subgroups $\mathcal{H}$ such that $H$ is hyperbolic relative to $\mathcal{H}$ we have
\[
 \stabdim(G;LC(\mathcal{G}))=\stabdim(H;LC(\mathcal{H}))<\infty.
\]
\end{btheorem}

The existence of such a collection of subgroups $\mathcal{H}$ is guaranteed by \cite{Behrstock-Drutu-Mosher}. Since the stable subsets we consider avoid peripherals, they embed quasi--isometrically into a coned--off graph. The fact that the relative stable dimension is finite in this case then follows from \cite[Theorem 17]{OsinOtherRelhyp}.

Again, we prove that this quasi--isometry invariant can be used to distinguish groups which are hyperbolic relative to quasi--isometric peripheral subgroups. 

\begin{btheorem}\label{bthm:relhypnonRH} Let $\mathcal H$ be a finite collection of finitely generated groups which are non-relatively hyperbolic. There is an infinite family of $1$--ended groups $G_n$, each of which is hyperbolic relative to $\mathcal H$, such that $\stabdim(G_n;LC(\mathcal{H}))\to\infty$ as $n\to\infty$.
\end{btheorem}

Combining the two constructions we obtain our most general statement.

\begin{btheorem}\label{bthm:relstabmixed} Let $G$ and $H$ be finitely generated groups which are hyperbolic relative to $\mathcal{G}$ and $\mathcal{H}$ respectively, such that each group in $\mathcal{G}\cup\mathcal{H}$ either has finite stable dimension, or has infinite stable dimension and is not relatively hyperbolic. If $G$ and $H$ are quasi--isometric, then
\[
 \stabdim(G;LC(\mathcal{G}_\infty))=\stabdim(H;LC(\mathcal{H}_\infty))<\infty
\]
where $\mathcal{G}_\infty$ (resp. $\mathcal{H}_\infty$) is the collection of $G'\in \mathcal{G}_\infty$ (resp. $H'\in \mathcal{H}_\infty$) with infinite stable dimension.
\end{btheorem}

There is a subtlety here: in proving this theorem we obtain a quasi--isometry invariant on the class of groups which are hyperbolic relative to a collection of subgroups which are non-relatively hyperbolic or have finite stable dimension. There is no reason to suspect that this class of groups is quasi--isometrically rigid.

Again we use Theorem \ref{bthm:make1endedRH} to build examples of groups which ensure that this invariant successfully distinguishes groups with the same peripherals.

\begin{btheorem}\label{bthm:relhyp} Let $\mathcal H$ be a finite collection of finitely generated groups which have finite stable dimension or are non-relatively hyperbolic. There is an infinite family of $1$--ended groups $G_n$, each of which is hyperbolic relative to $\mathcal H$, such that $\stabdim(G_n;LC(\mathcal{H}_\infty))\to\infty$ as $n\to\infty$.
\end{btheorem}

\subsection*{Plan of the paper}

After establishing some preliminaries on relatively hyperbolic groups and stability, we give the construction of the relative stable dimension and prove Theorems $\ref{bthm:relstabdimQI}$ and $\ref{bthm:relstabmixed}$ in Section \ref{sec:relhyp}. We dedicate Section $\ref{sec:smallcanc}$ to graphical small cancellation theory and the proof of Theorem $\ref{bthm:make1endedRH}$.
Finally, in Section $\ref{sec:thms}$, we prove Theorems \ref{bthm:relhypstable}, \ref{bthm:relhypnonRH} and \ref{bthm:relhyp}.

\subsection*{Acknowledgements} The authors would like to thank Ruth Charney for interesting discussions, and are grateful to Jason Manning for pointing out the reference \cite{GroffBB}.

\section{Preliminaries}

Notation: we denote the $A$--neighbourhood of a subset $Y$ of a metric space $(X,d)$ by
\[
 [Y]_A = \setcon{x\in X}{d(x,Y)\leq A}.
\]

\begin{definition} [Morse geodesics] A geodesic $\gamma$ in a metric space is said to be \textbf{Morse} if there exists a function $N=N(\lambda, \epsilon)$ such that for any $(\lambda, \epsilon)$-quasi-geodesic $\sigma$ with endpoints on $\gamma$, we have that $\sigma \subset [\gamma]_N$. We call the function $N$ a \textbf{Morse gauge} and say that $\gamma$ is $N$--Morse.
\end{definition}

Stable subsets were originally introduced by Durham--Taylor \cite{Durham-Taylor}, here we use an equivalent definition from \cite{CordesHume_stable}.

\begin{definition} [Stable subsets] Let $X$ be a geodesic metric space, let $Y$ be a quasi--convex subset of $X$ and let $N$ be a Morse gauge.

The subset $Y$ is $N$\textbf{--stable} if, for every pair of points $x,y\in Y$ there is an $N$--Morse geodesic $[x,y]\subseteq X$.

We say $Y$ is \textbf{stable} if there exists some Morse gauge $N$ such that $Y$ is $N$--stable. 
\end{definition}

As defined in \cite{CordesHume_stable}, the \textbf{stable dimension} of $X$ is the supremal asymptotic dimension of a stable subset of $X$.

\begin{definition}\label{defn:relhyp} Let $G$ be a finitely generated group and let $\mathcal{H}$ be a finite collection of subgroups of $G$. A \textbf{coned--off graph} $\hat{\Gamma}$ of $G$ with respect to $\mathcal{H}$ is a graph obtained from a Cayley graph $\Gamma$ of $G$ by attaching an additional vertex $v_{gH}$ for every left coset of each $H\in\mathcal{H}$ and adding an edge $(v_{gH},g')$ whenever $g'\in gH$.

We say $G$ is \textbf{hyperbolic relative to} $\mathcal{H}$ if the following two conditions are satisfied:
\begin{itemize}
\item Some (equivalently every) coned--off graph of $G$ is hyperbolic.
\item (Bounded Coset Penetration Property)  Let $\alpha,\beta$ be geodesics in $\hat\Gamma$ with the same endpoints and let $H\in \mathcal{H}$.  Then there exists a constant $c$ such that:
		\begin{enumerate}
		\item if $\alpha\cap gH\neq \emptyset$ but $\beta\cap gH=\emptyset$ for some $g\in G$, then the $\Gamma$-distance between the vertex at which $\alpha$ enters $gH$ and the vertex at which $\alpha$ exits $gH$ is at most $c$, and
		\item if $\alpha\cap gH\neq\emptyset$ and $\beta\cap gH\neq\emptyset$, and $\alpha$ (resp. $\beta$) first enters $gH$ at $\alpha_1$ (resp. $\beta_1$) and last exits $gH_i$ at $\alpha_2$ (resp. $\beta_2$), then $\alpha_j$ and $\beta_j$ are at a $\Gamma$-distance of at most $c$ from each other, for $j=1,2$.
		\end{enumerate}
\end{itemize}
\end{definition}

We state here some important results in the theory of relatively hyperbolic groups which will be necessary in our paper.

\begin{theorem}\label{thm:perhiperalsfingen}\cite{BowditchRelHyp}
If $G$ is hyperbolic relative to $\mathcal{H}$ then each $H\in\mathcal{H}$ is finitely generated.
\end{theorem}

\begin{theorem}\label{thm:peripheralsstable} \cite[Lemma 4.15]{drutu-sapir} Let $G$ be hyperbolic relative to $\mathcal{H}$. There exists a Morse gauge $N$ such that each $H\in \mathcal{H}$ is an $N$--Morse subset of $G$.

In particular, if $x,y\in H$ can be connected by an $N$--Morse geodesic in a Cayley graph $\Gamma_H$ of $H$, then they can be connected by an $N'$--Morse geodesic in a Cayley graph $\Gamma_G$ of $G$ where $N'$ depends on $N$, $\Gamma_H$ and $\Gamma_G$ but not on the choice of $x,y$.
\end{theorem}

\begin{theorem}\label{thm:removehypperipheral} If $G$ is hyperbolic relative to $\mathcal{H}$ and $H\in\mathcal{H}$ is hyperbolic, then $G$ is hyperbolic relative to $\mathcal{H}\setminus\set{H}$.
\end{theorem}
\begin{proof}
By \cite{OsinRelhyp}, $G$ is hyperbolic relative to $\mathcal{H}$ if and only if $G$ is relatively finitely presented with respect to $\mathcal{H}$ and has linear relative Dehn function. As all hyperbolic groups are finitely presented and have linear Dehn function, $G$ is relatively finitely presented with respect to $\mathcal{H}\setminus\set{H}$ and has linear relative Dehn function.
\end{proof}

\begin{theorem}\label{thm:distformula} Let $G$ be a group which is hyperbolic relative to $\mathcal{H}$. Let $S$ be a finite symmetric generating set of $G$ and let $\Gamma=\Cay(G,S)$. Let $\hat{\Gamma}$ be the coned--off graph of $\Gamma$.

There exist constants $D,M_0$ such that for all $M\geq M_0$ there exist constants $K\geq 1$, $C\geq 0$ satisfying the following:

For any $x,y\in G$ and any geodesic $\gamma$ from $x$ to $y$ in $\Gamma$
\[
 K^{-1} d_\Gamma(x,y) - C \leq d_{\hat{\Gamma}}(x,y) + \sum_{B\in LC(\mathcal{H})} \{\!\{\diam(\gamma\cap [B]_D)\}\!\}_M \leq Kd_\Gamma(x,y) + C,
\]
where $\{\!\{x\}\!\}_M=x$ if $x\geq M$ and $0$ otherwise.
\end{theorem}

The distance formula above is a combination of \cite{Sis-proj} Theorem $0.1$ and Lemma $1.15$. The purpose of Lemma $1.17$ is to show that (for $D=R(1,0)$) the values $\diam(\gamma\cap [B]_D)$ and $\diam(\pi_B(x)\cup\pi_B(y))$, where $\pi_B$ is a closest point projection, differ by at most a fixed constant.
\medskip

The main tool we require from \cite{CordesHume_stable} is the following

\begin{theorem}\label{thm:upbdRelhyp} Let $G$ be a group which is hyperbolic relative to $\mathcal{H}$. If $\stabdim(H)<\infty$ for all $H\in\mathcal{H}$ then $\max_{H\in\mathcal{H}}\stabdim(H)\leq \stabdim(G)<\infty$.
\end{theorem}

\section{Stable subsets of relatively hyperbolic groups}\label{sec:relhyp}

In this section we outline three constructions of stable subsets of a relatively hyperbolic group, depending upon the peripherals, which can be used to produce quasi--isometry invariants.

The three cases we will consider are as follows: $G$ is hyperbolic relative to $\mathcal{H}$ and
\begin{enumerate}
 \item each $H\in\mathcal{H}$ has finite stable dimension,
 \item each $H\in\mathcal{H}$ is not relatively hyperbolic,
 \item each $H\in\mathcal{H}$ has finite stable dimension or is not relatively hyperbolic.
\end{enumerate}

In each case the goal is, given a triple $(G,\mathcal{H},S)$ where $G$ is hyperbolic relative to $\mathcal{H}$ and $S$ is a finite symmetric generating set of $G$, produce a family of stable subsets $\mathcal{X}(G)$ of the Cayley graph of $G$ with respect to $S$ with bounded asymptotic dimension, such that, if $q \colon G\to G'$ is a quasi--isometry, then for each $X\in \mathcal{X}(G)$ there is some $X'\in\mathcal{X}(G')$ such that $q(X)\subseteq X'$.

It follows immediately from this condition that the maximum of the asymptotic dimensions of the sets $X\in\mathcal{X}$ is a quasi--isometry invariant.

We split this into three parts, corresponding to the three possible types of peripheral subgroups.

The first of these, when each $H\in\mathcal{H}$ has finite stable dimension, is exactly the ``universal'' collection of stable subsets constructed in \cite{CordesHume_stable}.

\subsection{Stable approximations}

Given a geodesic metric space $X$ and a point $e\in X$ we define the following collection of stable subspaces of $X$ indexed by Morse gauges $N$: $X^{(N)}_e$ is the set of all points $y\in X$ such that there exists an $N$--Morse geodesic $[e,y]$ in $X$.

From \cite[Theorem A]{CordesHume_stable} it follows that the $X^{(N)}_e$ are hyperbolic, stable, and that if $q:X\to Y$ is a quasi--isometry, then for all $N$ there exists some $N'$ such that $q(X^{(N)}_e)$ is a quasi-convex subset of $Y^{(N')}_{q(e)}$.

The \textbf{stable dimension} of $X$ is defined to be the supremum of the asymptotic dimension of $X^{(N)}_e$ over all Morse gauges $N$. This value is a quasi--isometry invariant.

Moreover, from \cite[Theorem I]{CordesHume_stable} we know that if $G$ is hyperbolic relative to $\mathcal{H}$ and each $H\in \mathcal{H}$ has finite stable dimension, then so does $G$.

In particular, if $H\in\mathcal{H}$ is hyperbolic, then it is a stable subgroup of $G$ \cite[Lemma $4.15$]{drutu-sapir}, so $\mathrm{asdim}(H)\leq \stabdim(G)$.

\subsection{Avoiding non-relatively hyperbolic peripherals}\label{sec:avoidnRH}

It is possible to proceed immediately to the general case, but to avoid a subtlety in the most general setting we first deal with the case where $G$ is hyperbolic relative to $\mathcal{H}$, and every $H\in \mathcal{H}$ is non-relatively hyperbolic.

Let $G$ be a finitely generated group and let $\mathcal{G}$ be a collection of subgroups of $G$. Let $N$ be a Morse gauge and let $D,L>0$. We equip $G$ with a word metric $d$ coming from a finite symmetric generating set.

We define $(G;LC(\mathcal{G}))^{(N)}_{D,L}$ to be the set of all points in $G$ which can be connected to the identity $e$ by an $N$--Morse geodesic $\gamma_x$ such that, for any $g\in G$ and $H\in\mathcal{G}$,
$\diam\left( \gamma_x \cap [gH]_D\right)\leq L$.

One key tool we will require is:

\begin{theorem} \cite[Theorem $4.1$]{Behrstock-Drutu-Mosher}

Let $K$ be a non-relatively hyperbolic group and let $G$ be hyperbolic relative to $\mathcal{H}$. If $q:K\to G$ is a quasi--isometric embedding, then there exists some $A>0$, $g\in G$ and $H\in\mathcal{H}$ such that $q(K)\subseteq [gH]_A$.
\end{theorem}

\begin{proposition}\label{prop:nRHqiinv} Let $G$ be a finitely generated group which is hyperbolic relative to a collection of non-relatively hyperbolic groups $\mathcal{G}$. Let $H$ be a finitely generated group. If $q:G\to H$ is a quasi--isometry, then for any collection $\mathcal{H}$ of non-relatively hyperbolic subgroups of $H$ such that $H$ is hyperbolic relative to $\mathcal{H}$ and every $N,D,L$ there exist $N',D',L'$ such that $q\left( (G;LC(\mathcal{G}))^{(N)}_{D,L}\right) \subseteq (H;LC(\mathcal{H}))^{(N')}_{D',L'}$.
\end{proposition}
By \cite{Behrstock-Drutu-Mosher} such a collection $\mathcal{H}$ always exists.

\begin{proof} Replacing $q$ by $q'(g)=q(e_G)^{-1}q(g)$ we may assume $q(e_G)=e_H$. We fix $K\geq 1$, $C\geq 0$ such that $q,q^{-1}$ are $(K,C)$-quasi--isometries and $d_X(x,q\circ q^{-1}(x))\leq C$ for all $x\in X$.

Notice that since the sets $X^{(N)}_{D,L}$ are nested, it suffices to prove the proposition for all $N,L$ and all sufficiently large $D$.

It is immediate from \cite[Theorem A]{CordesHume_stable} that for every $N$ there is some $N'$ such that $q\left( G^{(N)}_e\right) \subseteq H^{(N')}_e$, so it remains to show that coset intersections are controlled.

Let $x\in (G;LC(\mathcal{G}))^{(N)}_{D,L}$ and let $\gamma_x$ be an $N$--Morse geodesic from $e_G$ to $x$ such that, for any $g\in G$ and $H\in\mathcal{G}$,
$\diam\left( \gamma_x \cap [gH]_D\right)\leq L$. 

Set $y=q(x)$ and let $\gamma_y$ be an $N'$--Morse geodesic from $e_H$ to $y$. It follows that $q(\gamma_x)\subseteq [\gamma_y]_{N'(K,C)}$.

Suppose there exists a left coset $B\in LC(\mathcal{H})$ such that $y_1,y_2\in q(\gamma_x)\cap [B]_{KD+C}$. Note that, since $\gamma_y$ is $N'$--Morse, there exist points $y_1',y_2'\in\gamma_y$ such that $d_H(y_i,y_i')\leq N'(K,C)$.

Our goal is to find an upper bound on $d_H(y_1,y_2)$. 

Since $y_1,y_2\in q(\gamma_x)$, there must exist $x_1,x_2\in \gamma_x\cap [q^{-1}(B)]_{K^2D+KC+C}$ with the property that $d_G(x_1,x_2)>K^{-1}d_H(y_1,y_2)-C$. Since peripheral subgroups are not relatively hyperbolic, by \cite[Theorem $4.1$]{Behrstock-Drutu-Mosher} there exists some left coset $A$ of a subgroup in $\mathcal{G}$ such that $q^{-1}(B)\subseteq [A]_\lambda$ for some $\lambda$ which may be chosen independently of the coset $B$.

Applying \cite[Theorem $4.1$($\alpha_2'$)]{drutu-sapir} with $\theta=\frac{1}{3}$, we see that there is some uniform constant $M$ such that $\gamma_x\cap [A]_M\neq\emptyset$. Note that $M$ depends on $K,C,\lambda$, but it is independent of $D$.
By \cite[Lemma $4.15$]{drutu-sapir} any sub-quasi--geodesic of $q(\gamma_x)$ from $y_1$ to $y_2$ is contained in $[B]_{t(KD+N'(K,C)+C)}$ where $t$ depends on $K,C$ but not on $D$.

Combining these results, we see that there exists a constant $\kappa$ such that every sub-geodesic of length at least $\kappa D+\kappa$ on the restriction of $\gamma_x$ to a geodesic from $x_1$ to $x_2$ contains a point in $[A]_{M}$. Note that $\kappa$ can be chosen to be independent of $D$.

If $D\geq M$ then the diameter of $\gamma_x\cap [A]_D$ is greater than $d_G(x_1,x_2)-2(\kappa D+\kappa)$. This is a contradiction if $d_G(x_1,x_2)> L + 2(\kappa D+\kappa)$, so we deduce that $d_H(y_1,y_2)\leq KL + 2K\kappa D + 2K\kappa + KC$.

Therefore any two points $y'_1,y'_2$ on $\gamma_y$ contained in $[B]_{KD+C}$ are at distance at most  $KL + 2K\kappa D + 2K\kappa +2N'(K,C) + KC$.

The result follows by setting $D'=K\max\set{D,M}+C$ and $L'=KL + 2K\kappa \max\set{D,M} + 2K\kappa +2N'(K,C) + KC$.
\end{proof}

\begin{definition} Let $G$ be a finitely generated group and let $\mathcal{G}$ be a collection of subgroups of $G$. Let $X$ be a Cayley graph of $G$. The \textbf{stable dimension of $G$  relative to $LC(\mathcal{G})$} $\stabdim(G;LC(\mathcal{G}))$ is defined to be the supremum of the asymptotic dimensions of the sets $(G;LC(\mathcal{G}))^{(N)}_{D,L}$.
\end{definition}

\begin{corollary}\label{cor:relstdimQI}  Let $G$ be a finitely generated group which is hyperbolic relative to a collection of non-relatively hyperbolic groups $\mathcal{G}$. If $H$ is a finitely generated group and $g:G\to H$ is a quasi--isometry then for any collection $\mathcal{H}$ of non-relatively hyperbolic subgroups of $H$ such that $H$ is hyperbolic relative to $\mathcal{H}$ we have $\stabdim(G;LC(\mathcal{G}))=\stabdim(H;LC(\mathcal{H}))$.
\end{corollary}

\begin{proposition}\label{prop:nRHtoconedoff} Let $G$ be a finitely generated group which is hyperbolic relative to a collection of non-relatively hyperbolic groups $\mathcal{G}$. Let $X$ be a Cayley graph of $G$ and let $\hat{X}$ be the corresponding coned--off graph. For each $N,D,L$ there is a quasi--isometric embedding $(G;LC(\mathcal{G}))^{(N)}_{D,L}\to\hat{X}$.
\end{proposition}
\begin{proof}
This follows immediately from the distance formula for relatively hyperbolic groups \cite[Theorem $0.1$]{Sis-proj} and \cite[Lemma $3.4$]{CordesHume_stable}.
\end{proof}

\begin{corollary}\label{cor:relstabdimfinite} Let $G$ be a finitely generated group which is hyperbolic relative to a collection of non-relatively hyperbolic groups $\mathcal{G}$. Then $\stabdim(G;LC(\mathcal{G}))<\infty$.
\end{corollary}
\begin{proof} This follows immediately from Proposition \ref{prop:nRHtoconedoff} and \cite[Theorem 17]{OsinOtherRelhyp}.
\end{proof}

Theorem \ref{bthm:relstabdimQI} follows from Corollaries \ref{cor:relstdimQI} and \ref{cor:relstabdimfinite}.

\subsection{Mixed constructions}

We now make a related construction under the assumption that $G$ is hyperbolic relative to $\mathcal{G}_\infty\cup \mathcal{G}_F$ where each $G_\infty\in \mathcal{G}_\infty$ has infinite stable dimension but is non-relatively hyperbolic, and each $G_f\in\mathcal{G}_F$ has finite stable dimension.

We will focus on the collection of stable subsets $(G;LC(\mathcal{G})_\infty)^{(N)}_{D,L}$ introduced in Section \ref{sec:avoidnRH}. Here we are only avoiding the cosets of peripheral subgroups with infinite stable dimension.

\begin{proposition} Let $G,H$ be finitely generated groups which are hyperbolic relative to $\mathcal{G}=\mathcal{G}_\infty\cup\mathcal{G}_F$ and $\mathcal{H}=\mathcal{H}_\infty\cup\mathcal{H}_F$ respectively, where each group in $\mathcal{G}_\infty\cup\mathcal{H}_\infty$ has infinite stable dimension and is non-relatively hyperbolic, and each group in $\mathcal{G}_F\cup\mathcal{H}_F$ has finite stable dimension.

If $q:G\to H$ is a quasi--isometry, then for every $N,D,L$ there exist $N',D',L'$ such that $q:(G;LC(\mathcal{G}_\infty))^{(N)}_{D,L} \to (H;LC(\mathcal{H}_\infty))^{(N')}_{D',L'}$ is a quasi--isometric embedding.
\end{proposition}
\begin{proof} If $\Gamma\leq G$ is non-relatively hyperbolic and has infinite stable dimension, then $q(\Gamma)$ is contained in a uniform neighbourhood of some peripheral coset $h\Lambda$ of $H$ \cite{Behrstock-Drutu-Mosher}.  As quasi--isometries preserve stable subsets, and their asymptotic dimensions, such $\Lambda$ has infinite stable dimension. By assumption $\Lambda\in\mathcal{H}_\infty$. Let $q^{-1}$ be a quasi--isometric inverse of $q$, by the same logic we deduce that $q^{-1}(\Lambda)$ is contained in a uniform neighbourhood of some peripheral coset $g\Gamma$, so $q$ maps cosets of subgroups in $\mathcal{G}_\infty$ to (uniform neighbourhoods of) cosets of subgroups in $\mathcal{H}_\infty$.

Therefore the result follows from the argument in Proposition \ref{prop:nRHqiinv}.
\end{proof}

\begin{corollary}\label{cor:mixedstabQI} Let $G,H$ be finitely generated groups which are hyperbolic relative to $\mathcal{G}=\mathcal{G}_\infty\cup\mathcal{G}_F$ and $\mathcal{H}=\mathcal{H}_\infty\cup\mathcal{H}_F$ respectively, where each group in $\mathcal{G}_\infty\cup\mathcal{H}_\infty$ has infinite stable dimension and is non-relatively hyperbolic, and each group in $\mathcal{G}_F\cup\mathcal{H}_F$ has finite stable dimension.

Then $\stabdim(G;LC(\mathcal{G}_\infty))=\stabdim(H;LC(\mathcal{H}_\infty))$.
\end{corollary}

\begin{proposition}\label{prop:mixedstabdimfinite} Let $G$ be as above. Then $\stabdim(G;LC(\mathcal{G}_\infty))<\infty$.
\end{proposition}
\begin{proof}
We immediately deduce the following version of the distance formula (cf. Theorem \ref{thm:distformula}) for pairs of points in $(G;LC(\mathcal{G}_\infty))^{(N)}_{D,L}$. There exist constants $D,M_0$ such that for all $M\geq M_0$ there exist constants $K\geq 1$, $C\geq 0$ satisfying the following:

For any $x,y\in (G;LC(\mathcal{G}_\infty))^{(N)}_{D,L}$ and any geodesic $\gamma$ from $x$ to $y$ in $\Gamma$
\[
 K^{-1} d_G(x,y) - C \leq d_{\hat{\Gamma}}(x,y) + \sum_{B\in LC(\mathcal{H}_F)} \{\!\{\diam(\gamma\cap [B]_D)\}\!\}_M \leq Kd_\Gamma(x,y) + C,
\]
where $\{\!\{r\}\!\}_M=r$ if $r\geq M$ and $0$ otherwise.

Armed with this distance formula, we are able to apply the techniques of \cite{MackSis}, and deduce that $(G;LC(\mathcal{G}_\infty))^{(N)}_{D,L}$ quasi--isometrically embeds in the product of the coned--off graph $\hat{\Gamma}$ and a quasi--tree of spaces $\mathcal{C}$ (see \cite{BBF}) with pieces uniformly quasi--isometric to groups in $\mathcal{H}_F$. The results in \cite{Hume12} allow us to replace $\mathcal{C}$ by a tree-graded space with pieces uniformly quasi--isometric to groups in $\mathcal{H}_F$.

Let $\phi$ denote the map from $(G;LC(\mathcal{G}_\infty))^{(N)}_{D,L}$ to $\mathcal{T}$. We claim there exists some $N'$ such that 
$\phi\left((G;LC(\mathcal{G}_\infty))^{(N)}_{D,L}\right)\subseteq \mathcal{T}^{(N')}_{\phi(e)}$.

Given the claim, it is easy to see that $\phi\left((G;LC(\mathcal{G}_\infty))^{(N)}_{D,L}\right)$ is contained in a subset $\mathcal{T}^{(N)}$ of $\mathcal{T}$ which is tree-graded and the pieces are uniformly quasi--isometric to $H^{(N')}_e$ for $H \in \mathcal{H}_F$. By assumption, the asymptotic dimension of the $H^{(N')}_e$ is bounded independent of $N'$.

Thus the spaces $\mathcal{T}^{(N)}$ have uniformly bounded asymptotic dimension. Since the coned--off graph $\hat{G}$ has finite asymptotic dimension \cite{Bell-Fujiwara} we have a uniform bound on the asymptotic dimension of the $(G;LC(\mathcal{G}_\infty))^{(N)}_{D,L}$ is finite as required.

The proof of the claim follows from the same argument as in \cite[Theorem $8.3$]{CordesHume_stable}.
\end{proof}

Theorem \ref{bthm:relstabmixed} follows from Corollary \ref{cor:mixedstabQI} and Proposition \ref{prop:mixedstabdimfinite}.

\section{Small cancellation over free products}\label{sec:smallcanc}

Our next task is to prove that given any finite collection of finitely generated groups $\mathcal{H}$ there exists a $1$--ended group $G$ which is hyperbolic relative to $\mathcal{H}$. For this we will use small cancellation theory over free products (Theorem \ref{thm:1endedrelhyp}).

\begin{theorem}\label{thm:wildworldofrelhyp} Let $\mathcal{H}$ be a finite collection of finitely generated groups. There exists a $1$--ended group $G$ which is hyperbolic relative to $\mathcal{H}$.
\end{theorem}

The definitions and preliminary results in this section can all be found in \cite{Gruber-thesis} which gives a much more in--depth discussion.

Let $\setcon{G_i}{i\in I}$ be a set of non--trivial groups where each $G_i$ is generated by a symmetric set $S_i$ and let $\Gamma$ be a graph whose edges are oriented and labelled by elements of $S=\bigsqcup_{i\in I} S_i$. We denote an edge in $\Gamma$ as a set of two vertices $\set{x,y}$.
We simultaneously consider $\Gamma$ to be a graph with directed edges: for each oriented edge $\set{x,y}$ we associate two directed edges $e^+,e^-$, where $e^+$ is directed with the orientation and $e^-$ is directed against it.

\begin{definition}\label{defn:diagram} A \textbf{diagram} $D$ is a finite, simply connected 2-complex with a fixed embedding into the plane such that the image of every 1-cell, has an orientation and a label from a fixed set $S$. A \textbf{disc diagram} is a diagram homeomorphic to the 2-disc. The \textbf{boundary} of a disc diagram $D$, denoted $\partial D$, is its topological boundary inside $\R^2$. 
\end{definition}

We consider the $1$--skeleton of a diagram, $D^{(1)}$, as a planar graph with a fixed embedding into the plane. A \textbf{face} in $D$ is the closure of a bounded connected component of $\R^2\setminus D^{(1)}$. The boundary of a face is identified with a cycle subgraph of $D^{(1)}$. The boundary of the diagram is also considered as a subgraph of $D^{(1)}$.

The \textbf{degree} of a vertex $v$ in a graph $\Gamma$ or a diagram $D$ is the number of vertices $w$ such that $\set{v,w}$ is an edge. It is denoted $\mathrm{deg}(v)$.
\medskip

Given a directed edge $e$ in $\Gamma$ we denote its initial vertex by $\iota(e)$ and its terminal vertex by $\tau(e)$.

A \textbf{walk} in $\Gamma$ is a tuple of directed edges $W=(e_1,\dots,e_n)$ where $\tau(e_i)=\iota(e_{i+1})$ for all $1\leq i<n$. A walk is called a \textbf{path} if the vertices $\iota(e_1),\ldots,\iota(e_{n})$ are all distinct. A walk is \textbf{closed} if $\iota(e_1)=\tau(e_n)$.

The \textbf{label} of a directed edge $e_i$, denoted $l(e_i)$, is $s$ if $e_i=\set{\iota(e_i),\tau(e_i)}^+$ and $s^{-1}$ if $e_i=\set{\iota(e_i),\tau(e_i)}^-$.

The label of the walk $W=(e_1,\dots,e_n)$ is $l(W)=l(e_1)\dots l(e_n)$, which we consider as an element of the free monoid freely generated by $S$.

Given a diagram $D$, we define the \textbf{face graph} of $D$, $F_D$ to be the graph whose vertex set is the set of faces of $D$, with edges $\set{\Pi,\Pi'}$ whenever $\partial\Pi\cap\partial\Pi'$ contains an edge. Since $D$ is planar, $F_D$ is planar. Given any connected component $\Lambda$ of $F_D$, the union of the faces which define vertices in $\Lambda$ is a subdiagram of $D$ which is a disc diagram.

A path $P=(e_1,\dots,e_n)$ in a diagram $D$ is said to be \textbf{interior} if $\tau(e_i)\not\in\partial D$ for all $1\leq i < n$.

Abusing notation, we say that a subset $P$ of the $1$-skeleton of a diagram $D$ is a path $(e_1,\ldots,e_n)$ if the set of vertices of $P$ is $\set{\iota(e_1),\tau(e_1),\ldots,\tau(e_n)}$ and the set of edges of $P$ is $\setcon{\set{\iota(e_i),\tau(e_i)}}{1\leq i <n}$.
\medskip

Given a graph $\Gamma$ with edges labelled by elements in $S$, we define $G(\Gamma)$ to be the quotient of $*_{i\in I} G_i$ by the subgroup normally generated by all elements of $*_{i\in I} G_i$ which appear as labels of simple closed paths in $\Gamma$. Note that the label of a closed path is well-defined up to reduced cyclic conjugation and inversion, so any choice of representatives normally generates the same subgroup of $*_{i\in I} G_i$.
\medskip

To exclude unwanted cases we assume that all connected components of $\Gamma$ are finite.

\begin{definition}\label{defn:completion} Let $\Gamma$ be a graph whose edges are oriented and labelled by elements of the set $S$. The \textbf{completion} of $\Gamma$, denoted $\overline \Gamma$, is the graph with oriented edges labelled by elements of $S$ obtained via the following two-step procedure. 
\begin{itemize}
\item Onto every edge labelled by $s_i\in S_i$, attach a copy of $\Cay(G_i,S_i)$ along an edge of $\Cay(G_i,S_i)$ labelled by $s_i$. If, for some $i$, no $s_i\in S_i$ is the label of any edge of $\Gamma$, add a copy of $\Cay(G_i,S_i)$ as a separate component. Call this graph $\Gamma'$.
\item Take the quotient of $\Gamma'$ by the following equivalence relation: for edges $e$ and $e'$, we define $e\sim e'$ if $e$ and $e'$ have the same label and if there exists a path from the initial vertex of $e$, $\iota(e)$, to $\iota(e')$ whose label is trivial in $*_{i\in I}G_i$.
\end{itemize}
\end{definition}

\begin{definition}\label{defn:piece}
Let $\Gamma$ be an edge--labelled graph. A \textbf{piece} is a labelled oriented path $P$ admitting two distinct label and orientation--preserving graph homomorphisms $p\to\Gamma$.
\end{definition}

\begin{definition} Let $\lambda>0$. Let $\Gamma$ be labelled over $\sqcup_{i\in I}S_i$, where $S_i$ are generating sets of non-trivial groups $G_i$. We say $\Gamma$ satisfies the \textbf{$C_*'(\lambda)$-condition} if every attached $\Cay(G_i,S_i)$ in $\overline \Gamma$ is an embedded copy of $\Cay(G_i,S_i)$ and in $\overline \Gamma$ every piece $p$ that is locally geodesic and that is a subpath of a simple closed path $\gamma$ such that the label of $\gamma$ is non--trivial in $*_{i\in I}G_i$ satisfies $|p|<\lambda|\gamma|$.  
\end{definition}

The first result we will need is the following.

\begin{theorem}\cite[Theorem $2.9$]{Gruber-thesis} Let $H_1,\dots,H_k$ be groups generated by finite sets $S_i$. Let $\Gamma$ be a finite graph labelled over $S=\sqcup_{i\in I}S_i$ which satisfies the $C_*'(\frac16)$-condition. Then $G(\Gamma)$ is hyperbolic relative to $\set{H_1,\dots,H_k}$.
\end{theorem}

Now we work to ensure that such groups can be made $1$--ended by choosing $\Gamma$ appropriately. The goal is to generalise the method of Champetier \cite{ChampSC}, for this we will need a version of Strebel's classification of locally geodesic triangles in small cancellation groups over free products.

\begin{definition}
A \textbf{diagram over $\overline \Gamma$} is a diagram where every $2$--cell, $\Pi$, called a \textbf{face}, has boundary label equal to the label (as an element of the free monoid) of a simple closed path in $\overline\Gamma$ that is non--trivial in $*_{i\in I}G_i$, or $\Pi$ bears the label of a simple closed path in some $\Cay(G_i,S_i)$ and has no interior edge in the diagram.

We call faces of the first type \textbf{non--trivial}, and the second type \textbf{trivial}.
\end{definition}

Notice that a disc diagram over $\overline \Gamma$ consists either entirely of non--trivial faces or is a single trivial face.

\begin{lemma}[Curvature formula] \cite[Equation (8)]{Str90} Let $D$ be a disc diagram. Let $V$ be the set of vertices of $D$, $E$ the set of edges, and $F$ the set of faces. For each $B\in F$, let $e(B)$ be the number of edges  of $\partial B$ contained in $\partial B\cap\partial D$ and let $i(B)$ be the number of edges of $\partial B$ not contained in $\partial B\cap\partial D$. For each $k$ define $F_k=\setcon{B\in F}{e(B)=k}$.

Then,
\[
 6 = \sum_{v\in V} (3 - d(v)) + \sum_{B\in F_0} (6-i(B)) + \sum_{B\in F_0} (4-i(B)) + \sum_{k\geq 2}\sum_{B\in F_k} ((6-2k)-i(B)).
\]
\end{lemma}

\subsection{Combinatorics of reduced diagrams}

Given a word $w\in M(S)$, the free monoid freely generated by $S$, a $w$--diagram over $\overline \Gamma$ is a diagram over $\overline\Gamma$ where the external boundary has label $w$. We denote the length of a word $w\in M(S)$ by $\abs{s}$. Using Van Kampen's lemma, a $w$--diagram exists if and only if $w=_{G(\Gamma)} 1$. A $w$--diagram is said to be \textbf{minimal} if it has the minimal possible number of edges and amongst diagrams with this many edges it has the minimal number of vertices. Such a diagram is not necessarily unique. 
\medskip

The goal of this section is to prove that results on the structure of geodesic polygons in classical small cancellation groups also apply to geodesic polygons in small cancellation groups over free products.

\begin{definition} Let $D$ be a diagram over $\overline \Gamma$. An \textbf{arc} is a path $(e_1,\ldots,e_n)$ such that for all $1\leq i < n$, $d(\tau(e_i))=2$. An interior path which is also an arc is called an \textbf{interior arc}.
\end{definition}

\begin{theorem}\cite[Theorem $1.35$]{Gruber-thesis}\label{Grminwdiag} Let $\Gamma$ be a graph labelled over $S=\sqcup_{i\in I}S_i$ which satisfies the $C_*'(\frac{1}{6})$-condition. Let $w\in M(S)$ satisfy $w=_{G(\Gamma)} 1$, and let $D$ be a minimal $w$--diagram. Then every interior arc is a piece and any face with an interior edge has boundary labelled by a word which is non--trivial in $*_i G_i$.
\end{theorem}
\begin{remark} The theorem actually states that no interior edge \textbf{originates from} $\overline{\Gamma}$, but since the statements ``a path $P$ originates from $\overline{\Gamma}$'' and ``a path $P$ is an interior arc'' both pass to subpaths of $P$ we deduce that no interior arc originates from $\overline{\Gamma}$. We say nothing more about the definition of a path originating from $\overline{\Gamma}$ except that any interior arc of a diagram that does not originate from $\overline{\Gamma}$ is a piece. See \cite[Definition $1.22$]{Gruber-thesis} for details.
\end{remark}

\begin{corollary}\label{wdiagcomb} Let $\Gamma$ be a graph labelled over $S=\sqcup_{i\in I}S_i$ which satisfies the $C_*'(\frac{1}{6})$-condition. Let $w\in M(S)$ satisfy $w=_G 1$, and let $D$ be a minimal $w$--diagram.
\begin{itemize}
	\item[\textup{(1)}] If $\Pi,\Pi'$ are distinct faces in $D$ then $\partial\Pi\cap\partial\Pi'$ is a piece.
	\item[\textup{(2)}] If $\Pi$ is an interior face in $D$ ($\partial\Pi\cap\partial D$ does not contain an edge) then there exist at least $7$ other faces $\Pi_i$ in $D$ such that $\partial\Pi\cap\partial\Pi_i$ contains an edge.
	\item[\textup{(3)}] If $\Pi$ is a face in $D$ such that $\partial\Pi\cap\partial D$ is a path whose label defines a geodesic in $\Cay(G(\Gamma),S)$ then there exist at least $4$ other faces $\Pi_i$ in $D$ such that $\partial\Pi\cap\partial\Pi_i$ contains an edge.
\end{itemize}
\end{corollary}
\begin{proof}
We first show that (2) and (3) follow from (1).

Suppose for a contradiction that there are $k\leq 6$ other faces $\Pi_i$ in $D$ such that $\partial\Pi\cap\partial\Pi_i$ contains an edge. Note that by Theorem \ref{Grminwdiag} the faces $\Pi,\Pi_i$ are all non--trivial.

Then $\abs{\partial\Pi} = \sum_{i=1}^k \abs{\partial\Pi\cap\partial\Pi_i}$, so there exists some $i$ such that $\abs{\partial\Pi\cap\partial\Pi_i}\geq\frac16 \abs{\partial\Pi}$. By (1), $\partial\Pi\cap\partial\Pi_i$ is a piece, which contradicts the assumption that $\Gamma$ satisfies the $C_*'(\frac{1}{6})$-condition.

Similarly, for (3), note that if $\partial\Pi\cap\partial D$ is a path whose label defines a geodesic in $\Cay(G(\Gamma),S)$, then $\abs{\partial\Pi\cap\partial D}\leq\frac12\abs{\partial\Pi}$. Suppose for a contradiction that there are $k\leq 3$ other faces $\Pi_i$ in $D$ such that $\partial\Pi\cap\partial\Pi_i$ contains an edge.

Then $\frac12\abs{\partial\Pi} \leq \sum_{i=1}^k \abs{\partial\Pi\cap\partial\Pi_i}$ , so there exists some $i$ such that $\abs{\partial\Pi\cap\partial\Pi_i}\geq\frac16 \abs{\partial\Pi}$. By (1), $\partial\Pi\cap\partial\Pi_i$ is a piece, which contradicts the assumption that $\Gamma$ satisfies the $C_*'(\frac{1}{6})$-condition.

Now we prove (1). By Theorem \ref{Grminwdiag}, any interior arc contained in $\partial\Pi\cap\partial\Pi'$ is a piece.

Suppose for a contradiction that there exists a minimal $w$--diagram $D$ which contains two faces $\Pi,\Pi'$ such that $\partial\Pi\cap\partial\Pi'$ is not a piece. We fix such a $D$ with the minimal possible number of faces.

There are two ways in which $\partial\Pi\cap\partial\Pi'$ could fail to be a piece. It could fail to be a path. If it is a path, then it is an interior path -- by planarity, $\partial\Pi\cap\partial\Pi'\cap\partial D$ consists of at most $2$ vertices -- but it could fail to be an interior arc.

Suppose first that it fails to be a path. Let $\Pi,\Pi'$ be faces in $D$ such that $\partial\Pi\cap\partial\Pi'$ is not connected. Choose $x,y\in \partial\Pi\cap\partial\Pi'$ such that there exist interior paths $P_\Pi=(e_1,\ldots,e_n)$ and $P_{\Pi'}=(e'_1,\ldots,e'_m)$ with the following properties:
$\iota(e_1)=\iota(e'_1)=x$; $\tau(e_n)=\tau(e'_m)=y$; $\tau(e_i)\not\in\partial\Pi'$ for all $1\leq i < n$; and $\tau(e'_j)\not\in\partial\Pi$ for all $1\leq j < m$.

$D$ admits a subdiagram $D'$ which is a minimal $l(P_\Pi)l(P_{\Pi'})^{-1}$--diagram. By construction this diagram has a non-trivial face and has strictly fewer faces than $D$. The same is true of $D'\cup\Pi$ and $D'\cup\Pi'$, therefore the intersection of the boundaries of any two faces in $D'$, or of the boundary of a face in $D'$ with either $\partial\Pi$ or $\partial\Pi'$ is a piece.

Fix a non-empty connected component $\Gamma$ of $F_{D'}$ and let $D''$ be the disc subdiagram of $D'$ whose set of faces is precisely the vertex set of $\Gamma$.

From the small cancellation condition we see that every face $F$ in $D''$ with $e(F)\geq 1$ has at least $5$ neighbours in the face graph $F_{D''}$, so $i(F)\geq 5$. Moreover, if $e(F)=0$ then $i(F)\geq 7$. Inputting these values into the curvature formula we obtain a contradiction as the right hand side of the equation is strictly negative. 
\medskip

Now suppose that $\partial\Pi\cap\partial\Pi'$ is a path $P=(e_1,\ldots,e_n)$ but that there exists some $i$ with $1\leq i <n$ such that $d(\tau(e_i))\geq 3$. It follows that there is an edge $e$ which intersects the interior of either the face $\Pi$ or $\Pi'$. Without loss of generality assume it is $\Pi$. By minimality every edge lies on the boundary of a face, so $e\in\partial\Pi''\neq\Pi$, by planarity $\Pi''$ is contained in the interior of $\Pi$. Since $\partial\Pi\cap\partial\Pi'$ is a path, $\partial\Pi\cap\partial\Pi''$ is a single vertex $v$. Removing the edges and vertices of $\Pi''$ except for $v$ yields a new $w$--diagram with fewer edges, contradicting minimality.
\end{proof}

Parts (2) and (3) in the above corollary are exactly the hypotheses required to deduce Strebel's classification of diagrams whose boundary is a simple geodesic triangle in $\Cay(G(\Gamma),S)$ \cite[Theorem $43$]{Str90}. We highlight the specific case which will be necessary for this paper.

Every word in the free monoid uniquely determines a walk in $\Cay(G(\Gamma),S)$. If a word $w$ determines a walk which is a concatenation of $n$-geodesics $\gamma_1,\ldots,\gamma_n$ in $\Cay(G(\Gamma),S)$ then we say that $w$ is the \textbf{label of a geodesic $n$--gon}. The $\gamma_i$ are called the \textbf{sides} of the $n$--gon.

\begin{lemma}\label{relativeStrebelbigon} Let $\Gamma$ be a graph labelled over $S=\sqcup_{i\in I}S_i$ which satisfies the $C_*'(\frac{1}{6})$-condition. Let $w\in M(S)$ be the label of the boundary of a geodesic bigon in $\Cay(G(\Gamma),S)$ with sides $\gamma_1,\gamma_2$ and let $D$ be a minimal $w$--diagram over $\overline{\Gamma}$.

Every face $F$ in $D$ has its boundary contained in a union of four paths $P_1,\ldots$ where $P_1\subseteq\gamma_1$ and $P_2\subseteq\gamma_2$ have positive length, while $P_3$ and $P_4$, when they have positive length, are subpaths of the boundaries of two other faces in $D$.
\end{lemma}
The same conclusion can be drawn in the more general setting where $D$ is a minimal $w$--diagram, such that the boundary of $D$ is composed of two geodesics and a part of the boundary of a face in $D$ \cite[Lemma $3.12$]{MackCdim}. The conclusion of the lemma and this remark are illustrated in the following figure.

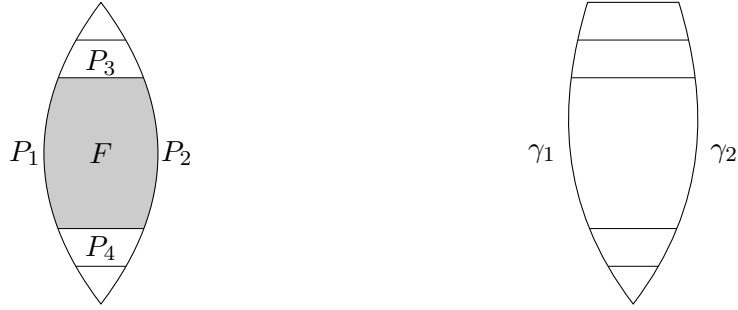
\begin{figure}[ht]\label{fig:generalisedbigons}
\centering
\begin{tikzpicture}[yscale=2,xscale=2, 
vertex/.style={draw,fill,circle,inner sep=0.3mm}]

\begin{scope}
\clip (0,3) .. controls (-0.5,2.33) and (-0.5,1.66) .. (0,1) .. controls (0.5,1.66) and (0.5,2.33) .. (0,3);

\fill[black!20!white]  (-2,2.5) -- (2,2.5) -- (2,1.5) -- (-2,1.5) -- (-2,2.5);

\draw[very thin]
				(-2,2.75) -- (2,2.75)  
				(-2,2.5) -- (2,2.5)
				(-2,1.5) -- (2,1.5)
				(-2,1.25) -- (2,1.25);
\end{scope}

\draw[thin] (0,3) .. controls (-0.5,2.33) and (-0.5,1.66) .. (0,1) .. controls (0.5,1.66) and (0.5,2.33) .. (0,3);

\path (0,2) node[] {$F$};

\path (-0.5,2) node[] {$P_1$};
\path (0.5,2) node[] {$P_2$};
\path (0,1.37) node[] {$P_4$};
\path (0,2.6) node[] {$P_3$};

\begin{scope}
\clip (3.2,3) .. controls (3,2.33) and (3,1.66) .. (3.5,1) .. controls (4,1.66) and (4,2.33) .. (3.8,3) -- (3.8,3);

\draw[very thin]
				(2.5,2.75) -- (4.5,2.75)  
				(2.5,2.5) -- (4.5,2.5)
				(2.5,1.5) -- (4.5,1.5)
				(2.5,1.25) -- (4.5,1.25);

\end{scope}

\draw[thin] (3.2,3) .. controls (3,2.33) and (3,1.66) .. (3.5,1) .. controls (4,1.66) and (4,2.33) .. (3.8,3) -- (3.8,3) -- (3.2,3);

\path (2.9,2) node[] {$\gamma_1$};
\path (4.1,2) node[] {$\gamma_2$};

\end{tikzpicture}\caption{The structure of bigons}
\end{figure}

\begin{remark} One can use Corollary \ref{wdiagcomb} to classify minimal diagrams whose boundary is a geodesic triangle using \cite{Str90} or a geodesic quadrangle using \cite{ACGH2}.
\end{remark}

This additional geometric structure in small cancellation groups over free products will enable us to construct one--ended relatively hyperbolic groups with any selected peripheral subgroups using techniques from \cite{ChampSC} (a closely related construction appears in \cite{MackCdim}). For everything we do in what follows it suffices to consider $\Gamma$ as a single cycle.

As a first step we require relative versions of Lemmas $4.19$ and $4.20$ in \cite{ChampSC}.

\begin{lemma}\label{lem:Champextgeodesics} Let $\Gamma$ be a graph with girth at least $7$ labelled over $S$ which satisfies the $C_*'(\frac16)$-condition. Let $G=G(\Gamma)$ and let $g\in G\setminus \set{1}$. Let $[1,g]$ be a geodesic in $X=\Cay(G,S)$ and suppose the last edge of this geodesic is labelled by $s\in S_i$.

There exists at most one $t\in S\setminus S_i$ such that $d(1,gt)\leq d(1,g)$.
\end{lemma}
\begin{proof} Let $t\in S\setminus S_i$ be such that $d(1,gt)\leq d(1,g)$, pick a geodesic $[gt,1]$ and let $D$ be a minimal diagram whose boundary is the geodesic triangle $[1,g]\cup [g,gt]\cup [gt,1]$. By assumption $g\not\in [gt,1]$.

Let $\Pi$ be the face whose boundary contains the edge $[g,gt]$. This face cannot be trivial as any trivial face must have its boundary contained in $[g,gt]\cup[gt,1]$ and therefore has an internal edge, which is a contradiction to the definition of a diagram over $\overline{\Gamma}$. By the classification above it follows that $\partial\Pi$ must also contain the edge $[gs^{-1},g]$. 

Applying Corollary $\ref{wdiagcomb}$, we deduce that $\partial \Pi$ intersects the boundary of at most one other face in a path of length less than $\frac{1}{6}\abs{\partial\Pi}$ and intersects the geodesic $[gt,1]$ in a path of length at most $\frac12\abs{\partial\Pi}$. By assumption, $\abs{\partial \Pi}\geq 7$, so the edge $[g,gt]$ has length less than $\frac{1}{6}\abs{\partial\Pi}$.
Therefore, the intersection of $\partial\Pi$ with $[1,g]$ is a subgeodesic containing $g$ of length at least $\frac16\abs{\partial\Pi}$. 

Suppose $t'\in S\setminus S_i$ also satisfies $d(1,gt')\leq d(1,g)$. By the same reasoning there is a non-trivial face $\Pi'$ such that the intersection of $\partial\Pi'$ with $[1,g]$ is a subgeodesic containing $g$ of length at least $\frac16\abs{\partial\Pi'}$.

Suppose for a contradiction that $t\neq t'$. Then the piece $\partial\Pi\cap\partial\Pi'$ has length at least $\frac16\min\set{\abs{\partial\Pi},\abs{\partial\Pi'}}$ which contradicts the small cancellation hypothesis.
\end{proof}

Recall that the girth of a graph $\Gamma$, denoted $g(\Gamma)$ is the length of the shortest simple cycle of positive length in $\Gamma$.

\begin{lemma}\label{lem:Champnolongrelation} Let $\Gamma$ be a graph with girth at least $12$ labelled over $S=\sqcup S_i$ which satisfies the $C_*'(\frac18)$-condition and, in addition, assume that if $\gamma$ is an embedded subpath of a simple cycle $C$ in $\Gamma$ and the label of $\gamma$ is contained in some $G_i$, then $\abs{\gamma}<\frac18\abs{C}$.

Let $G=\fpres{S}{\Gamma}$ and let $g\in G\setminus \set{1}$. Let $[1,g]$ be a geodesic in $X$ and suppose the last edge $e^-_g$ of this geodesic is labelled by $s\in S_i$.

There exists at most $1$ element $u\in G\setminus G_i$ such that: $d(1,gu)=d(1,g)+d(g,gu)=d(1,g)+3$, the first edge $e^+_g$ of a geodesic $[g,gu]$ is labelled by some $t\in S\setminus S_i$ and such that some geodesic $[gu,1]$ intersects the boundary of a non--trivial face $\Pi$ in a subgeodesic containing $gu$ with length at least $\frac14\abs{\partial\Pi} + 3$.
\end{lemma}
\begin{proof} Fix $g$ and suppose there is such an element $u$, let $\gamma_{gu}=[gu,1]$ be a geodesic and let $\Pi$ be a non--trivial face satisfying the above hypotheses.

Firstly, suppose $g\not\in\gamma_{gu}$.

The two geodesics $\gamma=[1,g]\cup [g,gu]$ and $\gamma_{gu}$ form a geodesic bigon. Let $D$ be a minimal diagram with this bigon as its boundary. Since $g\not\in\gamma_{gu}$ there exist faces $F^+,F^-$ in this diagram containing $e^+_g$ and $e^-_g$ respectively. We first prove that $F^+=F^-$.

Suppose not, then as both faces have internal boundary they are non--trivial. By Corollary \ref{wdiagcomb} and Lemma \ref{relativeStrebelbigon}, the internal boundary of $F^+$ in $D$ consists of at most two pieces, so it contributes less than $\frac14$ of the length of $\partial F^+$, moreover, $\partial F^+\cap\gamma\subset [g,gu]$ and since $\gamma_{gu}$ is a geodesic $\abs{\partial F^+\cap \gamma_{gu}}\leq \frac12\abs{\partial F^+}$. Combining these results we see that $\frac14\abs{\partial F^+}<\abs{\partial F^+\cap \gamma}\leq 3$ which contradicts the assumption that the girth of $\Gamma$ is at least $12$.

Now there is a face $F$ containing $e^+_g$ and $e^-_g$, as these edges are labelled by $t\not\in S_i$ and $s\in S_i$ this face is non--trivial and either shares no edge with another face in $D$ or is one of the end faces in a disc diagram of the type in Lemma \ref{relativeStrebelbigon}.

In either case $\abs{\partial F\cap \gamma_{gu}}>\frac38\abs{\partial F}$ and $d(\partial F\cap \gamma_{gu},gu)\leq 2$.

It follows that $\abs{\partial F\cap \partial \Pi}\geq \min \set{\frac14\abs{\partial\Pi},\frac38\abs{\partial F}-2}$ which implies that $F=\Pi$ by the small cancellation hypothesis.

\medskip
Assume now that we have two elements $u,v$ with corresponding geodesics $\gamma_u,\gamma_v$ and non--trivial faces $F_u,F_v$ satisfying the hypotheses of the lemma.

From the previous argument, we may assume that $g\in [gu,1]\cap[gv,1]$. If this is not the case it is immediate from the above argument that $F_u=F_v$ and therefore $u=v$. As a result $\partial F_u\cap \partial F_v$ contains a maximal subgeodesic $\gamma$ of $[gu,1]\cap[gv,1]$ including $g$. Let $g'$ be the closest point on $\gamma$ to $1$. 

By the small cancellation hypothesis $\abs{\gamma}<\frac18\min\set{\abs{\partial F_u},\abs{\partial F_v}}$, so $g'\neq 1$. Let $D$ be a minimal diagram whose boundary consists of the subgeodesics of $[gu,1]$ and $[gv,1]$ from $g'$ to $1$ and let $F$ be the face whose boundary contains $g'$.

Notice that $\abs{\partial F\cap [gu,1]},\abs{\partial F\cap [gv,1]}>\frac38\abs{\partial F}$.

Since there are no trivial faces whose boundary is contained in $\partial F_u\cup\partial F_v$, it follows that $\abs{\partial F\cap \partial F_u}>\frac18\abs{\partial F_u}$, so $F$ is a non--trivial face (by assumption, no subpath of length at least $\frac18\abs{\partial F_u}$ in $\partial F_u$ is labelled by an element of any $G_i$). Therefore $F=F_u$ by the small cancellation criterion. Likewise $F=F_v$ so $u=v$ as required.
\end{proof}

\begin{theorem}\label{thm:1endedrelhyp} Let $\Gamma$ be a graph labelled over $\sqcup_{i\in I}S_i$ with $\abs{I}\geq 3$ which satisfies the $C_*'(\frac18)$-condition, and, in addition, assume that if $\gamma$ is an embedded subpath of a simple cycle $C$ in $\gamma$ and the label of $\gamma$ is contained in some $G_i$, then $\abs{\gamma}<\frac18\abs{C}$. Assume also that every word of length at most $8$ in $*_i G_i$ appears as a labelled subpath of a non--trivial simple cycle in $\overline\Gamma$.

Then $G=\fpres{S}{\Gamma}$ is finite or $1$--ended.
\end{theorem}

For the proof we assume $G$ is infinite and prove that in this situation it is $1$--ended. For our applications, $G$ will always contain a non-elementary hyperbolic group as a subgroup, so is always infinite.

Following the strategy of \cite{ChampSC} (see also \cite[Section $5$]{MackCdim}), we will inductively construct paths connecting any two points at distance $d$ from $1$ in $G$ by a path which is disjoint from the closed ball $B_G(1;d-1)$.

Let $g,g'\in G$ with $d(1,g)=d(1,g')=1$. Let $C$ be a non--trivial simple cycle in $\overline{\Gamma}$ which contains a subpath labelled by $g'g^{-1}$. By \cite[Section $4.5$]{Gruber-Sisto} every connected component of $\overline{\Gamma}$ embeds isometrically in the Cayley graph $X$, therefore the path from $g$ to $g'$ along $C$ is a path which avoids $1$, which completes the base case of the induction. 

Now suppose any pair of elements $x',y'$ of $G$ of length $d$ from $1$ in $G$ can be connected by a path $P_d(x',y')$ disjoint from $B_G(1,d-1)$.

Let $x,y\in G$ be of length $d+1$ and let $x',y'$ respectively be any choice of elements on a geodesic from $x$ (resp. $y$) to $1$ of length $d$. We will build $P_{d+1}(x,y)$ from $P_d(x',y')$.
Denote by $P_d(x,y)$ the path from $x$ to $y$ obtained from $P_d(x',y')$ via the following procedure:
If $x\in P_d(x',y')$ remove the subpath of $P_d(x',y')$ from $x'$ to $x$, if $x\not\in P_d(x',y')$ add the edge $[x,x']$. Likewise, if $y\in P_d(x',y')$ remove the subpath of $P_d(x',y')$ from $y$ to $y'$, if $y\not\in P_d(x',y')$ add the edge $[y',y]$.

Consider the function $d(1,\cdot)$ along $P_{d}(x,y)$. The minimum possible value is $d$, if this is not attained then we set $P_{d+1}(x,y)=P_d(x,y)$ and we are done.

Now let $\set{P_1,\dots,P_k}$ be all the maximal subpaths of $P_{d}(x,y)$ such that $d(1,z)=d$ for all $z\in P_i$. We assume the $P_i$ are ordered and orientated as subpaths of $P_{d}(x,y)$. The $P_i$ may be isolated vertices.

We apply the following process to the $P_i$ in order.

If $P_i$ is a single vertex $z$ we let $z'_1,z'_2$ be the vertices immediately preceding and succeeding $z$ on $P_{d}(x,y)$ and notice that $d(1,z'_1)=d(1,z'_2)=d+1$. By Lemma \ref{lem:Champextgeodesics} and the fact that $\abs{I}\geq 3$, there is at least one way to extend a geodesic $[1,z'_i]$ to a geodesic $[1,a_i]$ where $d(1,a_i)=d+4$ such that this extension does not satisfy the conclusion of Lemma \ref{lem:Champnolongrelation}. Choose such $a_i$. Let $P'_i$ be the path of length $8$ in $\Gamma$ given by the geodesic from $a_1$ to $z$ and from $z$ to $a_2$.

There is a relator $R$ whose boundary contains $P'_i$. Remove $[z'_1,z]$ and $[z,z'_2]$ from $P_d(x,y)$ and replace them by the path from $z'_1$ to $z'_2$ in $\partial R$ of length $\abs{\partial R}-2$.

If $P_i$ is a path of length at least $1$ we label its vertices $z_1,z_2,\dots,z_k$, we let $z'_1,z'_k$ be the vertices immediately preceding and succeeding $P_i$ on $P_{d}(x,y)$ and notice that $d(1,z'_1)=d(1,z'_k)=d+1$. By Lemma \ref{lem:Champextgeodesics} and the fact that $\abs{I}\geq 3$ there is at least one way to extend the geodesic $[1,z_i]$ to a geodesic $[1,a_i]$ (containing $z'_i$ when $i=1,k$) where $d(1,a_i)=d+3$ ($d+4$ when $i=1,k$), which does not satisfy the conclusion of Lemma \ref{lem:Champnolongrelation}. Choose such geodesics $[1,a_i]$, let $y_i$ be the point on $[1,a_i]$ such that $d(1,y_i)=d+1$ and let $P'_i$ be the path of length $7$ or $8$ in $\Gamma$ given by the geodesic from $a_i$ to $z_i$, the edge $[z_i,z_{i+1}]$ and the geodesic from $z_{i+1}$ to $a_{i+1}$.

By assumption there is a relator $R_i$ whose boundary contains $P'_i$. For each $i$ in order, remove $[y_i,z_i]\cup[z_i,z_{i+1}]$ from $P_d(x,y)$ and replace them by the path from $y_i$ to $y_{i+1}$ in $\partial R_i$ of length $\abs{\partial R_i}-2$. For $i=k-1$, remove $[y_{k-1},z_{k-1}]\cup[z_{k-1},z_{k}]\cup [z_{k},z'_k]$ from $P_d(x,y)$ and replace them by the path from $y_{k-1}$ to $y_{k}$ in $\partial R_{k-1}$ of length $\abs{\partial R_{k-1}}-3$.

The resulting walk need not be a path, but contains a suitable path $P_{d+1}(x,y)$ obtained by removing all maximal closed subwalks.

Our goal is to prove that $P_{d+1}(x,y)$ does not contain a point $p$ at distance $d(1,p)\leq d$.

Notice that it suffices to prove this for any point added on the above relations which lie on the added relators between $a_{i}$ and $a_{i+1}$.

\begin{lemma}\label{lem:extendingoneended} Let $p\in P_{d+1}(x,y)$ lie on an added section of a relator $R_i$. Then any geodesic $[p,1]$ contains a vertex $z$ on $\partial R_i\cap P_d(x',y')$ satisfying $d(1,z)=d$. In particular $d(1,p)\geq d+1$.
\end{lemma}
\begin{proof} Suppose there is some geodesic $\gamma=[p,1]$ such that $\gamma\cap P_d(x',y')=\emptyset$. 

Now $p$ lies on some relation $R_i$ added in the process of building $P_{d+1}(x,y)$ between $a_i$ and $a_{i+1}$. Up to reversing the orientations of paths, we may assume $d(p,a_i)\leq d(p,a_{i+1})$.

Consider the closed walk consisting of $\gamma$, the geodesic $\gamma_a$ from $1$ to $a_i$ containing $z_i$ and the subgeodesic of $\partial R_i$ from $p$ to $a_i$, which we label $\gamma'$, and let $D$ be a minimal diagram with this boundary. By assumption $D$ has a disc component which is not just a trivial face since $\gamma\cap [z_i,a_i]=\emptyset$. Therefore we may assume $p$ lies on the boundary of a non-trivial face $F$ in $D$ and that $\partial F$ contains $[z_i,a_i]$.

We add the non--trivial face $R_i$ to this diagram and remove any trivial faces whose boundary is entirely contained in $\partial R_i\cup\partial F$. Applying the classification (Lemma \ref{relativeStrebelbigon}) and the small cancellation hypothesis, we see that $\abs{\partial F\cap\partial R_i}<\frac18\abs{\partial F}$, so $\abs{\partial F\cap \gamma_a}\geq \frac14\abs{\partial F}+3$. However, in the construction of the path $P_{d+1}(x,y)$ we explicitly chose geodesics $[z_i,a_i]$ so that no face $F$ with $a_i\in\partial F$ satisfies $\abs{\partial F\cap \gamma_a}\geq \frac14\abs{\partial F}+3$. This is a contradiction, completing the proof.
\end{proof}

The construction is illustrated by the following figure.

\begin{center}
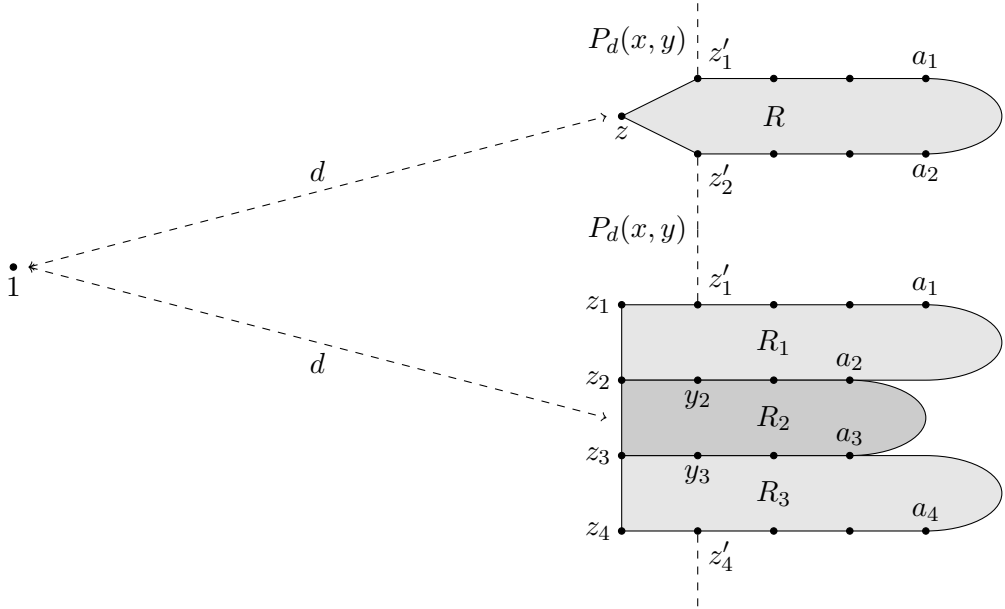
\begin{figure}[h!]
\begin{tikzpicture}[yscale=1,xscale=2, 
vertex/.style={draw,fill,circle,inner sep=0.3mm}]

\draw[<->, dashed] (-3.9,-2)--(-0.1,0);
\path (-2,-1) node[above] {$d$};

\filldraw[fill=black!10!white] 	(0,0) -- (0.5,0.5) --  (2,0.5) arc (90:-90:0.5 cm) -- (2,-0.5) -- (0.5,-0.5) -- (0,0);

\node[vertex]
(c) at ( -4, -2) {};
\path (c) node[below] {$1$};

\node[vertex]
(z) at ( 0, 0) {};
\path (z) node[below] {$z$};

\node[vertex]
(a) at ( 0.5, 0.5) {};
\path (a) node[above right] {$z'_1$};

\node[vertex]
(b) at ( 0.5, -0.5) {};
\path (b) node[below right] {$z'_2$};

\node[vertex]
(a1) at ( 1, 0.5) {};

\node[vertex]
(a1') at ( 1.5, 0.5) {};

\node[vertex]
(a2) at ( 2, 0.5) {};
\path (a2) node[above] {$a_1$};

\node[vertex]
(b1) at ( 1,-0.5) {};

\node[vertex]
(b1') at ( 1.5,-0.5) {};

\node[vertex]
(b2) at ( 2,-0.5) {};
\path (b2) node[below] {$a_2$};

\path (1,0) node[] {$R$};

\draw[dashed] 	(a) -- +(0,1)
				(b) -- +(0,-1);

\path (0.5,1) node[left] {$P_d(x,y)$};

\begin{scope}[yshift=-4cm]

\draw[<->, dashed] (-3.9,2)--(-0.1,0);
\path (-2,1) node[below] {$d$};

\filldraw[fill=black!10!white] 	
		(0,1.5) --  (2,1.5) arc (90:-90:0.5 cm) -- (2,0.5) -- (0,0.5) -- (0,1.5)
		(0,-0.5) --  (2,-0.5) arc (90:-90:0.5 cm) -- (2,-1.5) -- (0,-1.5) -- (0,-0.5);

\filldraw[fill=black!20!white] 	
		(0,0.5) --  (1.5,0.5) arc (90:-90:0.5 cm) -- (1.5,-0.5) -- (0,-0.5) -- (0,0.5);

\node[vertex]
(z1) at ( 0, 1.5) {};
\path (z1) node[left] {$z_1$};

\node[vertex]
(z2) at ( 0, 0.5) {};
\path (z2) node[left] {$z_2$};

\node[vertex]
(z3) at ( 0, -0.5) {};
\path (z3) node[left] {$z_3$};

\node[vertex]
(z4) at ( 0, -1.5) {};
\path (z4) node[left] {$z_4$};

\node[vertex]
(zz1) at ( 0.5, 1.5) {};
\path (zz1) node[above right] {$z'_1$};

\node[vertex]
(zz2) at ( 0.5, 0.5) {};
\path (zz2) node[below] {$y_2$};

\node[vertex]
(zz3) at ( 0.5, -0.5) {};
\path (zz3) node[below] {$y_3$};

\node[vertex]
(zz4) at ( 0.5, -1.5) {};
\path (zz4) node[below right] {$z'_4$};

\node[vertex]
(zzz1) at ( 1, 1.5) {};

\node[vertex]
(zzz2) at ( 1, 0.5) {};

\node[vertex]
(zzz3) at ( 1, -0.5) {};

\node[vertex]
(zzz4) at ( 1, -1.5) {};

\node[vertex]
(zzzz1) at ( 1.5, 1.5) {};

\node[vertex]
(zzzz4) at ( 1.5, -1.5) {};

\node[vertex]
(a1) at ( 2, 1.5) {};
\path (a1) node[above] {$a_1$};

\node[vertex]
(a2) at ( 1.5, 0.5) {};
\path (a2) node[above] {$a_2$};

\node[vertex]
(a3) at ( 1.5, -0.5) {};
\path (a3) node[above] {$a_3$};

\node[vertex]
(a4) at ( 2, -1.5) {};
\path (a4) node[above] {$a_4$};

\path (1,1) node[] {$R_1$};
\path (1,0) node[] {$R_2$};
\path (1,-1) node[] {$R_3$};

\draw[dashed] 	(zz1) -- +(0,1)
				(zz4) -- +(0,-1);

\path (0.5,2.5) node[left] {$P_d(x,y)$};

\end{scope}

\end{tikzpicture}\caption{Building the path $P_{d+1}(x,y)$ from $P_d(x,y)$.}
\end{figure}
\end{center}

\begin{proof}[Proof of Theorem $\ref{thm:wildworldofrelhyp}$]
In light of the work done in the rest of this section, it suffices to prove that there always exists a finite edge--labelled graph $\Gamma$ satisfying the hypotheses of Theorem \ref{thm:1endedrelhyp}. To ensure that each copy of $\Cay(G_i,S_i)$ is embedded it suffices to ensure that there is no simple closed loop in $\Gamma$ whose label is in $G_i\setminus\set{1}$.

If $\abs{\mathcal{H}}<3$ we add three copies of $\Z$ to $\mathcal{H}$ and continue.

Enumerate all words of length exactly $8$ in $*_{i\in I} G_i$ with respect to the generating set $S=\bigsqcup_{i\in I} S_i$ as $w_1,w_2,\dots,w_d$. Since $\abs{I}\geq 3$ and each $G_i$ is non--trivial, $d\geq 3.2^7$. 

For each $i$ fix some $s_i\in S_i\setminus\set{1}$. Without loss, assume that $w_1=(s_1s_2)^4$. 

Let $\Gamma$ be a single oriented cycle with label $w=\Pi_{i=1}^d w_1^i t_i w_i t'_i$ where $t_i$ is set to be $s_3$ unless $w_i$ begins with $s_3^{-1}$ in which case $t_i=s_1$, similarly $t'_i=s_3$ unless $w_i$ ends with $s_3^{-1}$ in which case we set $t'_i=s_2$. By construction $w$ is a reduced word in $*_{i\in I} G_i$.

It is clear that $\Gamma$ contains every possible labelled subpath of length $8$. If a subpath $\gamma$ of $\Gamma$ is labelled by a word in $G_i$ then $\abs{\gamma}\leq 8$ by construction.

Suppose that $p$ is a subpath of $\Gamma$ with $\abs{p}\geq 8d+20$, then $p$ contains a subword of the form $s_3w_1^is_3$, $s_3^{-1}s_2w_1^is_3$, $s_3w_1^is_1s_3^{-1}$ or $s_3^{-1}s_2w_1^is_1s_3^{-1}$ for some $i\geq 2$. In particular the gap between consecutive appearances of letters $s_3^{\pm 1}$ is between $8i$ and $8i+2$. By construction there is a unique gap of this length in $\Gamma$, but this length gap may also appear in the boundary word with the opposite orientation. Now assume $\abs{p}\geq 16d+40$, then $p$ contains two such gaps of length between $8i$ and $8i+2$, and $8i+8$ and $8i+10$ respectively, separated by between $8$ and $10$ places. This configuration is unique in $\Gamma$ with either orientation, which implies that $p$ is not a piece.

Thus any piece has length less than $16d+40$. The word $w$ has length at least $8\binom{d}{2}$. A simple calculation shows that the $C'(\frac18)$ graphical small cancellation condition is satisfied providing $d^2-33d-80>0$. In particular it is satisfied if $d\geq 36$, which is clearly true since $d\geq 3.2^7$.
\end{proof}

\section{Main Theorems}\label{sec:thms}

Here we combine the above results to prove Theorems \ref{bthm:relhypstable}, \ref{bthm:relhypnonRH} and \ref{bthm:relhyp}.

The final key ingredient is a family of hyperbolic groups with unbounded asymptotic dimension. All hyperbolic groups have finite asymptotic dimension \cite{Roe05}, however it is in general very difficult to construct a hyperbolic group of high dimension. There are a few constructions, the earliest of which we are aware is that of Januszkiewicz--{\'S}wi{\c{a}}tkowski, \cite{JanSwiCoxdim}, who construct virtually torsion--free hyperbolic Coxeter groups with arbitrarily high virtual cohomological dimension. By results of Bestvina--Mess \cite{BestMess} and Buyalo--Lebedeva \cite{BuyLeb}, the virtual cohomological dimension of a virtually torsion--free hyperbolic group equals its asymptotic dimension.
\medskip

Let $\mathcal{H}$ be a non-empty finite collection of finitely generated groups, possibly with repetitions.

We apply the construction in Theorem \ref{bthm:make1endedRH} to the sets $\mathcal{H}_n=\mathcal{H}\cup\set{H^n}$ where the $H^n$ are the hyperbolic groups of unbounded asymptotic dimension constructed in \cite{JanSwiCoxdim} (if $\abs{\mathcal{H}}=1$ add two distinct copies of $H^n$). From this we obtain a collection of $1$--ended groups $G^n$.

By Theorem $\ref{thm:removehypperipheral}$ each $G^n$ is hyperbolic relative to $\mathcal{H}$.

\begin{proof}[Proof of Theorem $\ref{bthm:relhypstable}$] If each $H\in\mathcal{H}$ has finite stable dimension, then the $G^n$ have unbounded but finite stable dimension. It follows from \cite[Corollary B]{CordesHume_stable} that there are infinitely many non-quasi isometric $1$--ended groups which are hyperbolic relative to $\set{H_i}$ completing the proof of Theorem \ref{bthm:relhypstable}.
\end{proof}
\medskip

\begin{proof}[Proof of Theorem $\ref{bthm:relhypnonRH}$] If every $H\in\mathcal{H}$ is non-relatively hyperbolic then considering each $G^n$ as a group hyperbolic relative to $\mathcal{H}$, we see that for each $n$ there exists some $N,D,L$ such that $H^n\subset (G^n;\mathcal{H})_{D,L}^{(N)}$, so $\stabdim(G;\mathcal{H)}\geq n$. By Corollary \ref{cor:relstabdimfinite}, $\stabdim(G;LC(\mathcal{H}))<\infty$. Therefore, Corollary \ref{cor:relstdimQI} implies that the collection of groups $G^n$ exhibit infinitely many different quasi--isometry types.
\end{proof}
\medskip

\begin{proof}[Proof of Theorem $\ref{bthm:relhyp}$] Set $\mathcal{H}\cup\set{H^n}=\mathcal{H}_F\cup \mathcal{H}_\infty$, where each $H\in\mathcal{H}_F$ has finite stable dimension and each $H\in \mathcal{H}_\infty$ has infinite stable dimension and is non-relatively hyperbolic.

Since $H^n$ has finite stable dimension, it belongs to $\mathcal{H}_F$ and therefore $\stabdim(G;LC(\mathcal{H}))\geq n$ as in the proof of Theorem \ref{bthm:relhypnonRH}. By Proposition \ref{prop:mixedstabdimfinite}, $\stabdim(G;LC(\mathcal{H})_\infty)<\infty$. Therefore, Corollary \ref{cor:mixedstabQI} implies that the collection of groups $G^n$ exhibit infinitely many different quasi--isometry types.
\end{proof}

{

}

\end{document}